\documentclass[12pt]{amsart}
\usepackage{SSdefn}

\author{Steven V Sam}
\address{Department of Mathematics, University of California, San Diego, CA}
\email{\href{mailto:ssam@ucsd.edu}{ssam@ucsd.edu}}
\urladdr{\url{http://math.ucsd.edu/~ssam/}}

\author{Andrew Snowden}
\address{Department of Mathematics, University of Michigan, Ann Arbor, MI}
\email{\href{mailto:asnowden@umich.edu}{asnowden@umich.edu}}
\urladdr{\url{http://www-personal.umich.edu/~asnowden/}}

\thanks{SS was partially supported by NSF DMS-1500069, DMS-1651327, and a Sloan Fellowship. AS was supported by NSF DMS-1453893.}

\title{Linear independence of powers}

\date{July 4, 2019}

\begin{document}

\maketitle

Fix an algebraically closed field $\bk$. We prove the following result:

\begin{theorem} \label{mainthm}
Let $R$ be an integral $\bk$-algebra, and let $f_1, \ldots, f_r \in R$ be non-zero elements such that $f_i/f_j \not\in \bk$ for all $i \ne j$. Then there exists an integer $1 \le e \le r!$ such that $f_1^e, \ldots, f_r^e$ are $\bk$-linearly independent.
\end{theorem}

\begin{remark}
We were motivated by \cite[Conjecture 16]{KTB}. This conjecture takes $R$ to be a polynomial ring over the real numbers, and asks for a bound $E$ depending on the number of variables and $r$ such that $f_1^e, \dots, f_r^e$ are linearly independent whenever $e \ge E$. Our methods do not seem able to obtain this result. We note that such a bound does not exist for general domains: consider the rings $R=\bk[x_1,\dots,x_d]/ (x_1^s + \cdots + x_d^s)$ with $f_i=x_i$, for example.
\end{remark}

We may as well replace $R$ with the subalgebra generated by the $f_i$'s. Thus, in what follows, we assume that $R$ is finitely generated. Thus $X=\Spec(R)$ is an integral scheme of finite type over $\bk$. If $R=\bk$ then the theorem is clear, so in what follows we assume $\dim(X) \ge 1$. The following is the key lemma:

\begin{lemma} \label{lem1}
Let $1 \le s \le r$ be given. There exist $\bk$-points $x_1, \ldots, x_s$ of $X$ such that the following two conditions hold:
\begin{enumerate}
\item $f_i(x_j) \ne 0$ for all $1 \le i \le r$ and $1 \le j \le s$.
\item Given $(i_1, \ldots, i_s) \ne (j_1, \ldots, j_s)$ in $[r]^s$, we have $f_{i_1}(x_1) \cdots f_{i_s}(x_s) \ne f_{j_1}(x_1) \cdots f_{j_s}(x_s)$.
\end{enumerate}
\end{lemma}

\begin{proof}
Let $Y$ be the open subvariety of $X$ where all the $f_i$'s are non-zero. We proceed by induction on $s$. The result is tautologically true for $s=0$. Suppose now that the result has been proven for $s-1$. Let $x_1, \ldots, x_{s-1}$ be the $\bk$-points witness this; note that these points all belong to $Y$. We now produce $x_s$.

For $i_{\bullet} \ne j_{\bullet} \in [r]^s$, let $U_{i_{\bullet},j_{\bullet}}$ be the locus of points $y \in Y$ such that
\begin{equation} \label{eq1}
f_{i_1}(x_1) \cdots f_{i_s}(y) \ne f_{j_1}(x_1) \cdots f_{j_s}(y).
\end{equation}
This is an open set. We claim that it is non-empty. There are two cases.

First, suppose that $i_s=j_s$. Then $(i_1,\ldots,i_{s-1}) \ne (j_1, \ldots, j_{s-1})$, and so $f_{i_1}(x_1) \cdots f_{i_{s-1}}(x_{s_1}) \ne f_{j_1}(x_1) \cdots f_{j_{s-1}}(x_{s-1})$ by assumption. Thus the two sides of \eqref{eq1} are different multiples of $f_{i_s}(y)=f_{j_s}(y)$, and so $U_{i_{\bullet},j_{\bullet}}=Y$.

Second, suppose that $i_s \ne j_s$. Then $f_{j_s}$ and $f_{i_s}$ are not scalar multiples of each other, by assumption, and so the two sides of \eqref{eq1} are unequal functions of $y$. Thus the claim follows.

Now let $U$ be the intersection of all the sets $U_{i_{\bullet},j_{\bullet}}$. This is a non-empty open set. We can take $x_s$ to be any $\bk$-point of it.
\end{proof}

We also require the following simple lemma:

\begin{lemma} \label{lem2}
Let $\alpha_1, \ldots, \alpha_t$ be distinct non-zero elements of $\bk$, and let $\beta_1, \ldots, \beta_t$ be elements of $\bk$ that are not all zero. Then there exists $1 \le j \le t$ such that $\sum_{i=1}^t \beta_i \alpha_i^j \ne 0$.
\end{lemma}

\begin{proof}
Let $A$ be the $t \times t$ matrix with entries $A_{i,j}=\alpha_i^j$ and let $B$ be the column vector with entries $B_i=\beta_i$. The determinant of $A$ is non-zero by the Vandermonde identity, and so $AB \ne 0$. Since the $j$th row of $AB$ is $\sum_{i=1}^t \beta_i \alpha_i^j$, the result follows.
\end{proof}

We can now prove the main result:

\begin{proof}[Proof of Theorem~\ref{mainthm}]
Let $x_1, \ldots, x_r \in X$ be the points produced by Lemma~\ref{lem1} with $s=r$. For $\sigma \in S_r$, let $c_{\sigma}=f_{\sigma(1)}(x_1) \cdots f_{\sigma(r)}(x_r)$. The $c_{\sigma}$ are distinct non-zero elements of $\bk$. Let $1 \le e \le r!$ be such that $\sum_{\sigma \in S_r} \sgn(\sigma) c_{\sigma}^e \ne 0$, which exists by Lemma~\ref{lem2}. For $1 \le i \le r$, let $v_i$ be the vector $(f_i^e(x_1), \ldots, f_i^e(x_r))$. These vectors are linearly independent, as the determinant of the matrix  with columns $v_1, \ldots, v_r$ is $\sum_{\sigma \in S_r} \sgn(\sigma) c_r$. It follows that the $f_1^e, \ldots, f_r^e$ are linearly independent, as a dependency would give one between the $v_i$'s.
\end{proof}

\begin{remark}
Suppose $\bk$ is not algebraically closed. Theorem~\ref{mainthm} remains true if we assume that $R$ is geometrically integral, i.e., that $\ol{\bk} \otimes_{\bk} R$ is integral. However, it is not true if we simply assume $R$ is integral. Indeed, if $R$ is a finite extension field of $\bk$ and $r>[R:\bk]$ then $f_1^e, \ldots, f_r^e$ are linearly dependent for all $e$, since any set of $r$ elements of $R$ is linearly dependent.
\end{remark}

\begin{remark}
The upper bound of $r!$ in Theorem~\ref{mainthm} is not optimal: for $r=3$, we can take $1 \le e \le 2$ in characteristic not~2, and $e \in \{1,3\}$ in characteristic~2. It is an interesting problem to determine the optimal upper bound on $e$.
\end{remark}

\end{document}